\documentclass[12pt,reqno]{amsart}
\usepackage[left=3cm,top=2cm,right=3cm,bottom=2cm]{geometry}
\usepackage{amssymb}
\usepackage{amsbsy}
\usepackage{epsfig}
\usepackage{tikz}
\usepackage{verbatim}
\usepackage[normalem]{ulem}
\usepackage[pdftex]{hyperref}
\hypersetup{colorlinks=true,linkcolor=blue,citecolor=blue,breaklinks = true}

\tikzstyle{vertex} = [fill,shape=circle,node distance=80pt]
\tikzstyle{edge} = [fill,opacity=.5,fill opacity=.5,line cap=round, line join=round, line width=40pt]
\tikzstyle{elabel} =  [fill,shape=circle,node distance=30pt]

\pgfdeclarelayer{background}
\pgfsetlayers{background,main}

\begin{document}
\title{Principal Eigenvector of the Signless Laplacian Matrix}
 	
\author[K. Cardoso]{Kau\^e Cardoso} \address{Instituto Federal do Rio Grande do Sul - Campus Feliz, CEP 95770-000, Feliz, RS, Brazil}
\email{\tt kaue.cardoso@feliz.ifrs.edu.br}
 	


\pdfpagewidth 8.5 in \pdfpageheight 11 in

\newcommand{\h}{\mathcal{H}}
\newcommand{\g}{\mathcal{G}}
\newcommand{\A}{\mathbf{A}}
\newcommand{\B}{\mathbf{B}}
\newcommand{\C}{\mathbf{C}}
\newcommand{\D}{\mathbf{D}}
\newcommand{\M}{\mathbf{M}}
\newcommand{\N}{\mathbf{N}}
\newcommand{\lin}{\mathcal{L}}
\newcommand{\cli}{\mathcal{C}}
\newcommand{\Q}{\mathbf{Q}}
\newcommand{\x}{\mathbf{x}}
\newcommand{\y}{\mathbf{y}}
\newcommand{\z}{\mathbf{z}}
\newcommand{\Ah}{\mathbf{A}(\mathcal{H})}

\theoremstyle{plain}
\newtheorem{Teo}{Theorem}[section]
\newtheorem{Lem}[Teo]{Lemma}
\newtheorem{Pro}[Teo]{Proposition}
\newtheorem{Cor}[Teo]{Corollary}

\theoremstyle{definition}
\newtheorem{Def}[Teo]{Definition}
\newtheorem{Afi}[Teo]{Claim}
\newtheorem{Que}[Teo]{Question}
\newtheorem{Exe}[Teo]{Example}
\newtheorem{Obs}[Teo]{Remark}

\newcommand{\keyword}[1]{\textsf{#1}}

\begin{abstract}
In this paper, we study the entries of the principal eigenvector of the signless Laplacian matrix of a hypergraph. More precisely, we obtain bounds for this entries. These bounds are computed trough other important parameters, such as spectral radius, maximum and minimum degree. We also introduce and study a new parameter related to edges of the hypergraph. This parameter is a spectral measure of a structural characteristic that can be thought of as an edge-variant of regularity. \newline

\noindent \textsc{Keywords.} Hypergraph; Signless Laplacian matrix; Principal eigenvector; Spectral radius.\newline

\noindent \textsc{AMS classification.} 05C65; 05C50; 15A18.
\end{abstract}

\maketitle

\section{Introduction}

Spectral graph theory analyzes the structure of graphs through the eigenvalues and eigenvectors
of matrices associated with them. Researchers, motivated by the success of this theory, have studied many
hypergraph matrices aiming to develop a spectral hypergraph theory. See for
example \cite{Feng, Reff2014, Rodriguez1}. In 2012, Cooper and Dutle \cite{Cooper} proposed the study of hypergraphs through tensors, and this new approach has been widely accepted by researchers of this area. However, to  obtain  eigenvalues of tensors has a high computational cost \cite{NP-hard}. In this regard, we see  that  the study of hypergraphs via matrices still has its place. Indeed, it worth pointing out that more recently, some authors have renewed the interest to study matrix representations of hypergraphs, as in \cite{Banerjee-matriz,Reff2019,kr-regular, matrix-pec-radius-2017,distance}.

The study of eigenvectors of a matrix has many applications. In particular, the principal eigenvectors of irreducible matrices form the basis of the PageRank algorithm used by Google \cite{google}. Moreover, the study of the principal eigenvector for graphs is interesting because the value of each of its entry may be seen as a spectral measure of the centrality of the vertex associated with this entry. Another interesting property of this vector is that the quotient and the subtraction of its two extreme entries can be understood as measurements of the irregularity (i.e., a measure of how close or far the hypergraph is to being regular).

In 2000, Papendieck and Recht \cite{vector1} obtained an upper bound for the maximal entry of the principal eigenvector of the adjacency matrix of a graph. In 2005, Zhang \cite{XZhang} investigate the ratio of any two extreme entries of the principal eigenvector of a graph. In 2007, Cioab\u{a} and Gregory \cite{Cioaba} improved some results of \cite{vector1} and \cite{XZhang}. In 2009, Das \cite{Das} obtained an upper bound for the maximal entry of the principal eigenvector of the signless Laplacian matrix of a graph and more recently Cardoso and Trevisan \cite{Kaue-vetor}, studied the principal eigenvector of the adjacency tensor of a hypergraph. Here we will study the principal eigenvector of the signless Laplacian matrix of a hypergraph. More specifically, we are interested in studying the extreme entries (i.e., the largest and the smallest entries) of this vector. Considering that there are few studies on this vector, some results that we prove in this paper are new even in the context of graphs.

Let $\h$ be a hypergraph whose incidence matrix is $\mathbf{B}(\h)$. The \emph{signless Laplacian matrix} of $\h$ is defined as $\mathbf{Q}(\h) = \mathbf{B}\mathbf{B}^T$. This matrix has already been studied in \cite{Kaue-energia,Kaue-lap}. The main goal of this paper is the study the principal eigenvector of this matrix. Theorems such as Perron-Frobenius and Rayleigh Principle can be inherited directly from matrix theory. In this way, we have the basic tools to study the spectral radius and the principal eigenvector of this matrix. 

In the last section of this paper, we call the attention to various hypergraph parameters which are yet to be explored. When $\h$ is a $k$-uniform hypergraph we define, for each edge $e=\{v_1,\ldots,v_k\}$, the number $\x(e) = x_{v_1}+\cdots+x_{v_k}$, where $\x=(x_v)$ is the principal eigenvector of $\Q(\h)$. Thus, we define  $\x(\max)$ and $\x(\min)$ as the largest and the smallest value reached by $\x(e)$ and $\Gamma(\h)$ as the quotient between $\x(\max)$ and $\x(\min)$. Remarkably, this last parameter provides the following characterization of a hypergraph which can be thought of as an edge-variant of regularity.

\vspace{0.1cm}\noindent\textbf{Theorem \ref{teo:caracG}.}\textit{	Let $\h$ be a uniform hypergraph. $\Gamma(\h) = 1$ if, and only if, for each edge the sum of the degrees of its vertices is constant.}\vspace{0.1cm}

We observe that if a hypergraph $\h$ has the parameter $\Gamma(\h)$ greater than 1, then the sum of the vertices in each edge is not constant, so we can say that $\Gamma(\h)$ is a measure of the distribution of degrees of vertices along the edges of the hypergraph.

The remaining of the paper is organized as follows. In Section \ref{sec:pre},
we present some basic definitions about hypergraphs and matrices, and prove some technical lemmas. In Section \ref{sec:bounds}, we study the entries of the principal eigenvector of $\Q$, obtaining bounds for them. In Section \ref{sec:edges}, we will prove Theorem \ref{teo:caracG} and obtain some results about the new parameters defined in this section.

\section{Preliminaries}\label{sec:pre}
In this section, we shall present some basic definitions about hypergraphs and matrices, as well as terminology, notation and concepts that will be useful in our proofs.
\vspace{0cm}

A \textit{hypergraph} $\h=(V,E)$ is a pair composed by a set of vertices $V(\h)$ and a set of (hyper)edges $E(\h)\subseteq 2^V$, where $2^V$ is the power set of $V$. $\h$ is said to be a $k$-\textit{uniform} (or a $k$-graph) for $k \geq 2$, if all edges have cardinality $k$. Let $\mathcal{H}=(V,E)$ and $\mathcal{H}'=(V',E')$ be hypergraphs, if $V'\subseteq V$ and $E'\subseteq E$, then $\mathcal{H}'$ is a \textit{subgraph} of $\h$.

The \textit{open neighborhood} of a vertex $v\in V(\h)$, denoted by $N(v)$, is the set formed by all vertices, distinct from $v$, that have some edge in common with $v$. The \textit{neighborhood} of $v\in V(\h)$ is defined as $N[v]$ = $N(v)\cup \{v\}$. The \textit{edge neighborhood} of a vertex $v\in V$, denoted by $E_{[v]}$, is the set of all edges that contain $v$. 

The \textit{degree} of a vertex $v\in V$, denoted by $d(v)$, is the number of
edges that contain $v$. More precisely, $\;d(v) = |E_{[v]}|$. A hypergraph is
$r$-\textit{regular} if $d(v) = r$ for all $v \in V$. We define the
\textit{maximum}, \textit{minimum} and \textit{average} degrees,
respectively, as \[\Delta(\h) = \max_{v \in V}\{d(v)\}, \quad \delta(\h) =
\min_{v \in V}\{d(v)\}, \quad d(\h) = \frac{1}{n}\sum_{v \in V}d(v).\]

Let $\h$ be a hypergraph. A \textit{walk} of length $l$ is a sequence of
vertices and edges $v_0e_1v_1e_2 \ldots e_lv_l$ where $v_{i-1}$ and $v_i$ are
distinct vertices contained in $e_i$ for each $ i=1,\ldots,l$. The
\textit{distance} between two vertices is the length of the shortest walk
connecting these two vertices. The \textit{diameter} of the hypergraph is the
largest distance between two of its vertices. The hypergraph is
\textit{connected}, if for each pair of vertices $ u, w$ there is a walk
$v_0e_1v_1e_2 \cdots e_lv_l $ where $ u = v_0 $ and $ w = v_l $. Otherwise,
the hypergraph is \textit{disconnected}.

A \textit{multigraph} is an ordered pair $\g = \left(V, E\right)$, where $V$
is a set of vertices and $E$ is a multi-set of pairs of distinct, unordered
vertices, called edges. Its \textit{adjacency matrix} $\A(\g)$, is the square
matrix of order $|V|$, where $a_{ii}=0$ and if $i \neq j$, then $a_{ij}$ is
the number of edges connecting the vertices $ i $ and $ j $.

For a hypergraph $\h$, its \textit{clique multigraph} $\cli(\h)$, has the same vertices as $\h$. The number of edges between two vertices of this multigraph is equal the number of hyperedges containing them in $\h$. For more details see \cite{Kaue-lap}.

\begin{Exe}
	The hypergraph $\h=(\{1,\ldots,5\}, \;\{123,145,345\})$ and its clique multigraph $\cli(\h)=(\{1,\ldots,5\}, \;\{12, 13, 14, 15, 32, 34, 35, 45, 45\})$ are illustrate in Figure \ref{fig:multigrafos}.
	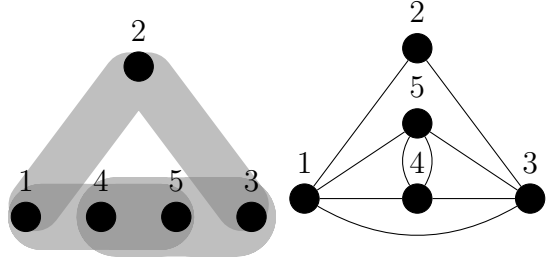
\begin{figure}[h!]	
		\centering

		\begin{tikzpicture}
		\node[draw,circle,fill=black,label=below:,label=above:\(1\)] (v1) at (0,0) {};
		\node[draw,circle,fill=black,label=below:,label=above:\(2\)] (v2) at (1.5,2) {};
		\node[draw,circle,fill=black,label=below:,label=above:\(3\)] (v3) at (3,0) {};
		\node[draw,circle,fill=black,label=below:,label=above:\(4\)] (v4) at (1,0) {};
		\node[draw,circle,fill=black,label=below:,label=above:\(5\)] (v5) at (2,0) {};
		
		\begin{pgfonlayer}{background}
		\draw[edge,color=gray,line width=20pt] (v1) -- (v2) -- (v3);
		\draw[edge,color=gray,line width=25pt] (v5) -- (v1) -- (v4) -- (v5);
		\draw[edge,color=gray,line width=30pt] (v4) -- (v3) -- (v5) -- (v4);		
		\end{pgfonlayer}				
		\end{tikzpicture}
		\begin{tikzpicture}	
		\node[draw,circle,fill=black,label=below:,label=above:\(1\)] (v1) at (0,0) {};
		\node[draw,circle,fill=black,label=below:,label=above:\(2\)] (v2) at (1.5,2) {};
		\node[draw,circle,fill=black,label=below:,label=above:\(3\)] (v3) at (3,0) {};
		\node[draw,circle,fill=black,label=below:,label=above:\(4\)] (v4) at (1.5,0) {};
		\node[draw,circle,fill=black,label=below:,label=above:\(5\)] (v5) at (1.5,1) {};
		
		\path
		(v1) edge node[below]{} (v2)
		(v2) edge node[below]{} (v3)
		(v3) edge [bend left] node[below]{} (v1)
		(v1) edge node[below]{} (v4)
		(v1) edge node[below]{} (v5)
		(v4) edge [bend left] node[below]{} (v5)
		(v4) edge [bend right] node[below]{} (v5)
		(v4) edge node[below]{} (v3)
		(v3) edge node[below]{} (v5);
		\end{tikzpicture}
	
		\caption{~The hypergraph $\h$ and its clique multigraph $\cli(\h)$.}\label{fig:multigrafos}	
	\end{figure}
\end{Exe}

Let $\M$ be a positive irreducible square matrix, its \textit{spectral radius} $\rho(\M)$ is its largest eigenvalue and its \textit{principal eigenvector} $\x$ is the positive normalized eigenvector from Perron-Frobenius Theorem. We call the pair $(\rho,\x)$, as \textit{principal eigenpair}. Let $\h = (V, E)$ be a hypergraph. The \textit{incidence matrix} $\B(\h)$ is defined as the matrix of order $|V|\times|E|$, where $\;b(v,e) = 1$ if  $v \in e$ and $\;b(v,e) = 0$ otherwise. Its \textit{matrix of degrees} $\D(\h)$ is a square matrix of order $|V|$, where $d_{ii} = d(i)$, and $d_{ij} = 0$ if $i \neq j$. Recall that the \textit{signless Laplacian matrix} is defined as $\Q = \B\B^T$.  We say that the spectral radius, the principal eigenvector and the principal eigenpair of $\Q$, are the \textit{spectral radius}, the \textit{signless Laplacian vector} and the \textit{signless Laplacian eigenpair} of $\h$, respectively.

\begin{Lem}[\cite{Kaue-lap}]\label{teo:multigrafo}	
	Let $\h$ be a $k$-graph, $\B$  its incidence matrix, $\D$ its matrix of degrees and $\A_\cli$ the adjacency matrix of its clique multigraph. So, we have $\;\B\B^T =\D+\A_\cli.$
\end{Lem}

For a non-empty subset of vertices $\alpha = \{v_1,\ldots,v_t\} \subset V$ and a vector $\x=(x_i)$ of dimension $n=|V|$, we denote $\x(\alpha)=x_{v_1}+\cdots+x_{v_t}$, so we can write
\[(\Q\x)_u = (\D\x)_u+(\A_\cli\x)_u = d(u)x_u+\sum_{e \in E_{[u]}}x\left(e-\{u\}\right)  = \sum_{e \in E_{[u]}}\x(e), \quad \forall u \in V(\h).\]

\begin{Lem}[\cite{Kaue-lap}]\label{Lem:xqx}
	Let $\h$ be a $k$-graph with $n$ vertices. For each vector $\x\in\mathbb{R}^n$, we have
	\[\x^T\Q\x=\sum_{e \in E}\left[ \x(e)\right]^2.\]
\end{Lem}

\begin{Lem}[\cite{Kaue-lap}]\label{teo:qregular}
	Let $\h$ be a connected $k$-graph with $n$ vertices and let $\rho(\h)$ be its spectral radius. The following statements are equivalent:
	\begin{itemize}
		\item[(a)] $\h$ is regular.
		\item[(b)] $\rho(\h) = kd(\h)$.
		\item[(c)] $\rho(\h) = k\Delta(\h)$.
		\item[(d)] The signless Laplacian vector of $\h$ is $\x=\left(\frac{1}{\sqrt{n}},\ldots,\frac{1}{\sqrt{n}}\right)$.
	\end{itemize}
\end{Lem}

\begin{Lem}[\cite{Kaue-lap}]\label{lem:graumaxmin}
	If $\h$ is a connected $k$-graph and $\rho(\h)$ is its spectral radius, then
	\[kd(\h) \leq \rho(\h) \leq k\Delta(\h).\]
\end{Lem}

\subsection{Technical lemmas}
In this subsection, we will prove some technical lemmas that will be useful in our proofs.

Let $\h$ be a hypergraph and $\alpha=\{v_ {i_1},\ldots, v_ {i_r}\} \subset V(\h)$ be a subset of vertices. We define the \textit{degree of a set} $\alpha$, as the number of edges that contain simultaneously all vertices of $\alpha$ and denote it for $ d(\alpha)$.

\begin{Lem}\label{lema:km}
	Let $\h=(V,E)$ be a $k$-graph. If $v \in V$, then $\sum_{u \in N[v]}d(\{v,u\}) = kd(v)$.
\end{Lem}
\begin{proof}		
	Just observe that, this sum is equal to the sum of the elements on the row associated with vertex $v$ in the signless Laplacian matrix.
\end{proof}

\begin{Lem}\label{lema:cota_para_soma}
	Let $r \leq s$ be integers and $f:\mathbb{R}_{++}^n \rightarrow \mathbb{R}$ the function defined by $f(\x) = \sum_{i = 1}^nx_i^{r}$. If $f(\x)$ is subject to the constraint $\sum_{i = 1}^nx_i^{s}=1$, then $f(\x) \leq \sqrt[s]{n^{s-r}}$. The equality holds if, and only if, $\x=\left(\frac{1}{\sqrt[s]{n}}, \ldots, \frac{1}{\sqrt[s]{n}}\right)$.
\end{Lem}
\begin{proof}
	Let $g(\x) = \sum_{i = 1}^nx_i^{s}-1$ and $L(\x,\gamma) = f(\x) - \gamma g(\x)$. We will compute the maximum value of $f(\x)$ using Lagrange multipliers:
	$$\begin{cases}
	L(\x,\gamma)_i=0,\, \forall\, i \in [n], \\
	L(\x,\gamma)_\gamma=0. \\
	\end{cases}$$
	
	\noindent From the derivatives in $ x_i $, we have
	\[rx_i^{r-1}-\gamma sx_i^{s-1}=0 \quad \Rightarrow \quad x_i = \sqrt[s-r]{\frac{r}{s\gamma}},\quad \forall \, i \in [n].\]
	Now, from the derivative of $\gamma$, we obtain the original constraint $\sum_{i = 1}^nx_i^{s}=1$. So,
	\[\sum_{i = 1}^n\left( \sqrt[s-r]{\frac{r}{s\gamma}} \right)^{s}=1 \Rightarrow  n\left(\sqrt[s-r]{\frac{r}{s\gamma}}\right)^s = 1  \Rightarrow  \sqrt[s-r]{\frac{r}{s\gamma}} = \frac{1}{\sqrt[s]{n}} \Rightarrow x_i = \frac{1}{\sqrt[s]{n}},\; \forall \, i \in [n].\]
	Thus, $\hat{\x} = \left(\frac{1}{\sqrt[s]{n}}, \cdots, \frac{1}{\sqrt[s]{n}}\right)$ maximizes $f(\x)$. Now, we will determine what this maximum value is:
	\[f(\hat{\x}) = \sum_{i =1}^n\left( \frac{1}{\sqrt[s]{n}}\right)^{r} = \frac{n}{(\sqrt[s]{n})^{r}} = \sqrt[s]{n^{s-r}}.\]
	Therefore, the result follows.
\end{proof}

\begin{Cor}\label{coro:cota_para_soma}
	Let $\h$ be a connected $k$-graph. If $\x$ is the signless Laplacian vector of $\h$, then
	\[1 \leq \sum_{v \in V}x_v \leq \sqrt{n}.\]
	The first equality is true if, and only if, $\h$ has a single vertex and the second is true if, and only if, $\h$ is regular.
\end{Cor}
\begin{proof}
	If $\h$ has only one vertex $v$, then $x_v=1$. Assume that $\h$ has more than one vertex, so $0 < x_v < 1$ for each $v \in V$. Thus, $x_v > x_v^{2}$, therefore $$\sum_{v \in V}x_v > \sum_{v \in V}x_v^2 = 1.$$
	The second inequality follows from Lemma \ref{lema:cota_para_soma}, setting $s=2$ and $r=1$.
\end{proof}

\begin{Lem}\label{lem:rho-kd}
	Let $\h$ be a connected $k$-graph. If $ (\rho, \x) $ is its signless Laplacian eigenpair, then $$\rho\sum_{v \in V}x_v = k\sum_{v \in V}d(v)x_v.$$
\end{Lem}
\begin{proof}
	Let $u$ be a vertex, note that $\rho x_u = \sum_{e \in E_{[u]}}\x(e)$, summing over the set of vertices we obtain
	\begin{eqnarray}
	\rho\sum_{v \in V}x_v &=& \sum_{v \in V}\left(\sum_{e \in E_{[v]}}\x(e) \right) = \sum_{v \in V}\left(\sum_{u \in N[v]}d(\{v,u\})x_v \right) \notag\\
	&=& \sum_{v \in V}x_v\left(\sum_{u \in N[v]}d(\{v,u\})\right) \stackrel{\mathrm{Lemma}\; \ref{lema:km}}{=} k\sum_{v \in V}d(v)x_v.\notag
	\end{eqnarray}
	Therefore, the result follows.
\end{proof}

\section{Bounds for signless Laplacian vector entries}\label{sec:bounds}
In this section, we obtain bounds for the extreme entries of the principal eigenvector of the signless Laplacian matrix, these bounds are computed using important classical and spectral parameters. In addition, we study  inequalities involving the ratio and difference between the two extreme entries of this vector.

\begin{Def}	
	Let $\h$ be a connected $k$-graph and $\x=(x_v)$ its signless Laplacian vector. We define
$$x_{\min} = \min_{ v \in V(\h)}\{x_v\}, \quad  x_{\max} = \max_{ v \in V(\h)}\{x_v\}, \quad \gamma(\h) = \frac{x_{\max}}{x_{\min}}.$$
\end{Def}

\begin{Obs} Let $\h$ be a connected $k$-graph. By Lemma \ref{teo:qregular}, we observe that $x_{\min} \leq \frac{1}{\sqrt{n}} \leq x_{\max}$, with equality only when $\h$ is regular.\end{Obs}

In Theorem 21 of \cite{Kaue-lap}, the authors obtain a bound for diameter of a hypergraph as a function of spectral radius, second largest eigenvalue and smallest entry of the principal eigenvector of $\Q$. With some algebraic operations, we obtain the following upper bound for the smallest entry of the signless Laplacian vector.

\begin{Cor}
	Let $\h$ be a connected $k$-graph with more than one edge.  If $ (\rho, \x) $ is its signless Laplacian eigenpair and the diameter of $\h$ is $D$, then
	$$x_{\min}\leq\frac{1}{\sqrt{1+\left(\frac{\rho}{\lambda_2} \right)^{D-1}}},$$ where $\lambda_2$ is the second largest eigenvalue of $\Q$.
\end{Cor}

\begin{Teo}\label{teo:cotaxi}
	Let $\h$ be a connected $k$-graph with $n$ vertices. If $ (\rho, \x) $ is its signless Laplacian eigenpair, then
	$$x_i \leq \frac{\sqrt{k}d(i)}{\rho}.$$
\end{Teo}
\begin{proof}
	Let $u \in V$ be a vertex, so $\rho x_u = \sum_{e \in E_{[u]}}\x(e)$, by cauchy-schwarz inequality we have $\rho^2 x_u^2 \leq d(u)\sum_{e \in E_{[u]}}(\x(e))^2$. Therefore
	\begin{eqnarray}
	\frac{\rho^2x_u^2}{d(u)} &\leq& \sum_{e \in E_{[u]}}(\x(e))^2 \leq  \sum_{e \in E_{[u]}}k(x_{v_{e1}}^2+\cdots+x_{v_{ek}}^2)\notag \\
	&=& k\sum_{v \in N[u]}d(\{u,v\})x_v^2 \leq kd(u)\sum_{v \in N[u]}x_v^2 \leq kd(u). \notag
	\end{eqnarray}
	Isolating $x_u$ we obtain the result.
\end{proof}

\begin{Obs}
	Let $\h$ be a connected $k$-graph and $u \in V$ be a vertex. By Theorem \ref{teo:cotaxi}, we have that $x_{\min}\leq \frac{\sqrt{k}\delta}{\rho}$ and $x_{\max}\leq \frac{\sqrt{k}\Delta}{\rho}$.
\end{Obs}

Let $\h$ be a hypergraph. Its \textit{Zagreb index} is defined as the sum of the squares of the degrees of its vertices. More precisely, $$Z(\h) = \sum_{v \in V(\h)}d(v)^2.$$ This is an important parameter in graph theory, having chemistry applications, \cite{energy-Zagreb}.

\begin{Teo}\label{txm} Let $\h$ be a connected $k$-graph with $n$	vertices. If $ (\rho, \x) $ is its signless Laplacian eigenpair, then	$$x_{\max}\geq \displaystyle\frac{\rho}{k\sqrt{Z(\h)}}.$$ The equality is achieved if, and only if, $\h$ is regular.
\end{Teo}
\begin{proof}
	We observe that $\rho x_u  = \sum_{e \in E_{[u]}}\x(e)\leq kd(u)x_{\max}$, hence $\rho^2 x_u^{2} \leq k^2d(u)^2x_{\max}^{2}$. Summing over the of vertices, we have \[\rho^2\sum_{u \in V}
	x_u^{2} \leq  \sum_{u \in V}k^2d(u)^2x_{\max}^{2}\quad \Rightarrow
	\quad \rho^2 \leq x_{\max}^2\left( k^2\sum_{u \in
		V}d(u)^2\right).\]
	Observe that equality occurs if, and only if, $ x_u = x_{\max} $ for all $ u \in V $. That is, equality occurs only when $ \h $ is regular.
\end{proof}

\begin{Teo}Let $\h$ be a connected $k$-graph with $n$ vertices. If $ (\rho, \x) $ is its signless Laplacian eigenpair, then	$$x_{\min}\leq \sqrt{\frac{k\delta^2}{\rho^2+k\delta^2(n-\delta)}}.$$
	The equality is achieved if, and only if, $\h$ is regular.
\end{Teo}
\begin{proof}
	Let $u$ be a vertex with minimum degree, by cauchy-schwarz inequality we have that
	\begin{eqnarray}
	(\rho x_u)^2 = \left(\sum_{e \in E_{[u]}}\x(e)\right)^2 \leq \delta \sum_{e \in E_{[u]}}(\x(e))^2 \leq \;k\delta\!\!\!\!\!\!\!\!\!\sum_{e=\{v_1,\ldots,v_k\} \in E_{[u]}}\!\!\!\!\!\!\!\!(x_{v_1}^2+\cdots+x_{v_k}^2)\notag\\ 
	= k\delta\left(\sum_{v \in N[u]}d(\{u,v\})x_v^2 \right)\leq k\delta^2\left(\sum_{v \in N[u]}x_v^2 \right) = k\delta^2\left(1 - \sum_{v \in V\smallsetminus N[u]}x_v^2\right). \notag  
	\end{eqnarray}
	Therefore $(\rho x_{\min})^2 \leq k\delta^2(1 - (n-\delta)x_{\min}^2).$
	Simplifying, we obtain the bound.
	
	Now, observe that if the equality is achieved, then $x_u = x_{\min}$ and  by the second inequality we have that if $v$ is a neighbor of $u$, then $x_v = x_u$. Even more, if $w$ is not a neighbor of $u$, then in the last inequality we have that $x_w = x_{\min}$. Thus, the eigenvector has all entries equals and therefore $\h$ must be regular.
\end{proof}

\begin{Teo} Let $\h$ be a connected $k$-graph with $n$ vertices and $m$ edges. If $ (\rho, \x) $ is its signless Laplacian eigenpair, then
	$$(kn\Delta - k^2m)x_{\min} \leq (k\Delta - \rho)\sqrt{n}.$$
	The equality is achieved if, and only if, $\h$ is regular.
\end{Teo}
\begin{proof}
	From Lemma \ref{lem:rho-kd}, we have 
	$$(k\Delta - \rho)\sum_{v \in V}x_v = \sum_{v \in V}(k\Delta - kd(v))x_v \geq x_{\min}\sum_{v \in V}(k\Delta - kd(v)).$$
	
	Now, by Corollary \ref{coro:cota_para_soma}, we have
	$$x_{\min}(kn\Delta - k^2m) \leq (k\Delta - \rho)\sqrt{n}.$$
	
	The equality is true if, and only if, $\sum_{v \in V}x_v = \sqrt{n}$, that is only when $\h$ is regular.
\end{proof}

We observe that the parameter $\gamma(\h)$ can be used to measure the irregularity of the hypergraph $\h$, because $\gamma(\h) = 1$ if, and only if, the hypergraph is regular.

\begin{Teo}
	Let $\h$ be a connected $k$-graph. If $ (\rho, \x) $ is its signless Laplacian eigenpair, then
	$$\gamma(\h) \geq \frac{\left(\rho - k\delta\right)\left(\Delta - d\right)}{\left(k\Delta - \rho\right)  \left(d-\delta\right)}.$$
	The equality is achieved if, and only if, $\h$ is regular or semi-regular where if $d(v) = \Delta$, then $x_v = x_{\max}$ and if $d(u) = \delta$, then $x_u = x_{\min}$. 
\end{Teo}
\begin{proof}
	By Lemma \ref{lem:rho-kd} we know that 
	\begin{equation}(k\Delta - \rho)\sum_{u \in V}x_u = k\sum_{u \in V}(\Delta - d(u))x_u \geq kn(\Delta - d)x_{\min}.\label{eq:teoG1}\end{equation}
	Similarly, we have
	\begin{equation}(\rho - k\delta)\sum_{u \in V}x_u = k\sum_{u \in V}(d(u) - \delta)x_u \leq kn(d - \delta)x_{\max}.\label{eq:teoG2}\end{equation}
	Multiplying (\ref{eq:teoG1}) and (\ref{eq:teoG2}), we conclude that
	$$\left( (k\Delta - \rho)\sum_{u \in V}x_u\right) \left( kn(d - \delta)x_{\max}\right) \geq \left((\rho - k\delta)\sum_{u \in V}x_u \right) \left(kn(\Delta - d)x_{\min} \right).$$
	Simplifying the above inequality, we obtain the bound.
	
	Observe that, the equality is achieved if, and only if, it holds in (\ref{eq:teoG1}) and (\ref{eq:teoG2}). Thus, all vertices that has no maximum degree must have minimum value in signless Laplacian vector and  all vertices that has no minimum degree must have maximum value in signless Laplacian vector. Therefore	$\h$ must be regular or semi-regular where if $d(v) = \Delta$, then $x_v = x_{\max}$ and if $d(u) = \delta$, then $x_u = x_{\min}$. 
\end{proof}

\begin{Teo}\label{TeoDd}
	Let $\h$ be a connected $k$-graph. If $(\rho, \x)$ is its signless Laplacian eigenpair, then
	\[\gamma(\h)\geq \max\left\lbrace
	\frac{k\Delta}{\rho},
	\frac{\rho}{k\delta}\right\rbrace \geq \sqrt{\frac{\Delta}{\delta}}.\]
	
	The equality is achieved if, and only if, $\h$ is regular.
\end{Teo}
\begin{proof} Let $u,v \in V$ be vertices, such that $d(u) = \Delta$ and $d(v) = \delta$,
	so we have
	\[\rho x_{\max}\geq \rho x_u =  \sum_{e\in
		E_{[u]}}\x(e)\geq \sum_{e\in E_{[u]}}kx_{\min}
	= k\Delta x_{\min}  \Rightarrow
	\frac{x_{\max}}{x_{\min}}\geq \frac{k\Delta}{\rho}.\]	
	\[\rho x_{\min}\leq \rho x_v =  \sum_{e\in E_{[v]}}\x(e)\leq \sum_{e\in E_{[v]}}kx_{\max} = k\delta x_{\max} \Rightarrow \frac{x_{\max}}{x_{\min}}\geq \frac{\rho}{k\delta}.\]

	For second inequality, we note that $\sqrt{\frac{\Delta}{\delta}}$ is the geometric mean between $ \frac{k\Delta}{\rho}$ and $\frac{\rho}{k\delta}$.
	
	Now, to achieve the equality the equalities bellow must be true
	\[\rho x_{\max} = \rho x_u\; \Leftrightarrow\; x_u = x_{\max}\; \textrm{ and } \; \sum_{e \in E_{[u]}}\x(e)=\Delta kx_{\min}\;\Leftrightarrow\;  x_p = x_{\min},\; \forall p \in N[u].\]
	That is, $x_{\max} = x_u = x_{\min}$, therefore $\h$ is regular. 
\end{proof}

\begin{Teo}
	Let $\h$ be a connected $k$-graph with $n$ vertices. If $\x$ is its signless Laplacian vector, then
	\[x_{\max} - x_{\min} \geq \frac{\sqrt{\Delta} - \sqrt{\delta}}{\sqrt{n\Delta}}.\]
	The equality is achieved if, and only if, $\h$ is regular.
\end{Teo}
\begin{proof} Firstly, we observe that
	\[\frac{x_{\max}}{x_{\min}}\geq \sqrt{\frac{\Delta}{\delta}} \Rightarrow x_{\min} \leq\sqrt{
		\frac{\delta}{\Delta}} x_{\max}.\]
	Multiplying the inequality by $ -1 $ and adding $ x_ {\max} $  to both sides,
	we have
	\[x_{\max} - x_{\min} \geq \left(1 - \sqrt{\frac{\delta}{\Delta}}\right) x_{\max} =
	\frac{\sqrt{\Delta} - \sqrt{\delta}}{\sqrt{\Delta}}x_{\max}. \] We observe that
	$x_{\max} \geq \frac{1}{\sqrt{n}}$, so we conclude that
	\[x_{\max} - x_{\min} \geq \frac{\sqrt{\Delta} - \sqrt{\delta}}{\sqrt{n\Delta}}.\]
	Equality occurs if, and only if, $x_{\max} = \frac{1}{\sqrt{n}}$, that is, only when $\h$ is regular.
\end{proof}

\begin{Teo}\label{tminmax}	Let $\h$ be a connected $k$-graph with $n$
	vertices. If $\x$ is its signless Laplacian vector, then	
	\begin{itemize}	
		\item[(a)] $x_{\max} \geq  \frac{\Delta}{\sqrt{\delta^2 +(n-1)\Delta^2}}$, equality holds if, and only if, the hypergraph $\h$ is regular.
		
		\item[(b)] $x_{\min} \leq  \frac{\delta}{\sqrt{\Delta^2+(n-1)\delta^2}}$, equality holds if, and only if, the hypergraph $\h$ is regular.
	\end{itemize}
\end{Teo}
\begin{proof}
	
	To prove part (a), we observe that \[1 = \sum_{u \in V} x_v^2 \leq x_{\min}^2
	+ (n-1)x_{\max}^2 =(\gamma^{-2} +n-1)x_{\max}^2 \leq
	\left(\left(\frac{\delta}{\Delta} \right)^2+n-1
	\right)x_{\max}^2. \]	
	Now observe that equality in this theorem occurs only if the equality from
	Theorem \ref{TeoDd} occurs as well. Therefore, the equality is achieved if, and only if, $\h$ is regular.
	
	Similarly, we prove part (b).
\end{proof}

\section{A Measure for the irregularity of hyperedges}\label{sec:edges}
In this section, we will study the parameter $\Gamma$ that are still little explored, even for other graph matrices.
We believe that the parameter $\Gamma$ plays the role for the edges of the hypergraph, played by $\gamma$ for vertices, which can be used as a measure of irregularity.

\begin{Def}
	Let $\h$ be a uniform hypergraph and $\x=(x_v)$ its signless Laplacian vector. We define
	\[\x(\min) = \min_{e \in E(\h)}\{\x(e)\}, \quad \x(\max) = \max_{e \in E(\h)}\{\x(e)\},\quad \Gamma(\h) = \frac{\x(\max)}{\x(\min)}.\]			
\end{Def}

\begin{Teo}\label{Teominrhormax}
	Let $ \h $ be a connected $k$-graph with $n$ vertices and $m$ edges. If $(\rho, \x)$ is its
	signless Laplacian eigenpair, then
	\[\x(\min) \leq \sqrt{\frac{\rho(\h)}{m}} \leq \x(\max).\]	
	If $\h$ is regular, then the equalities are achieved.
\end{Teo}
\begin{proof}
	By Lemma \ref{Lem:xqx}, we have that	
	\[\rho(\h) = \x^T\Q\x = \sum_{e \in E}(\x(e))^2 \leq \sum_{e \in E}(\x(\max))^2 = m(\x(\max))^2. \]
	
	The proof of the other inequality is analogous.
	
	If $\h$ is regular by Lemma \ref{teo:qregular}, we have $\x(\min) = \frac{k}{\sqrt{n}} = \x(\max)$ and $\rho = kd(\h) = \frac{k^2m}{n}$. Therefore the equalities are achieved.
\end{proof}

\begin{Cor} Let $\h$ be a connected $k$-graph with $n$ vertices. If $ (\rho, \x) $ is its signless Laplacian eigenpair, then	$$x_{\max}\geq \sqrt{\frac{\rho}{k^2m}}.$$ The equality is achieved if, and only if, $\h$ is regular.	
\end{Cor}
\begin{proof}
	We just notice that	$kx_{\max} \geq \x(\max) \geq \sqrt{\frac{\rho}{m}}.$	
	
	In proof of Theorem \ref{Teominrhormax}, we see that the equality holds if, and only if, $\x(\max) = \x(e)$ for all $e \in E$. Therefore, the equality in this corollary is achieved if, and only if, $kx_{\max}  = \x(e)$ for all $e \in E$, i.e. $x_u = x_{\max}$ for all $u \in V$, that is only when $\h$ is regular. 
\end{proof}
\begin{Obs}
	It is also possible to obtain the following inequality  $x_{\min}\leq \sqrt{\frac{\rho}{k^2m}}$,
	but it is not interesting because $x_{\min}\leq\frac{1}{\sqrt{n}}\leq \sqrt{\frac{\rho}{k^2m}}$.
\end{Obs}

\begin{Teo}\label{lemapeso}	
	Let $\h$ be a connected $k$-graph. If $(\rho, \x)$ is its signless Laplacian
	eigenpair, then
	\begin{itemize}
		\item[(a)] $kx_{\min}\leq \x(\min) \leq \frac{\rho}{\delta}x_{\min},$
		\item[(b)] $\frac{\rho}{\Delta}x_{\max}\leq \x(\max) \leq kx_{\max}.$
	\end{itemize}
	If $\h$ is regular, then the equalities are achieved.
\end{Teo}
\begin{proof}
	For part (a) notice that, for every $v \in V$, we have $x_{\min} \leq x_v \leq
	x_{\max}$, thus given an edge $e = \{v_1,\cdots, v_k\} \in E$, we have
	\[\x(e) = x_{v_1}+\cdots +x_{v_k} \geq  kx_{\min} \quad \forall e \in E  \quad \Rightarrow \quad
	\x(\min)\geq kx_{\min}.\]
	Let $u \in V$ such that $x_u = x_{\min}$, so
	\[\rho x_{\min} = \sum_{e \in E_{[u]}}\x(e) \geq \sum_{e \in E_{[u]}}\x(\min) \geq
	\delta \x(\min).\]
	
	Similarly, we prove part (b).
	
	If $\h$ is regular, then $k\delta  = \rho = k\Delta$. Therefore, the equalities hold.
\end{proof}

\begin{Cor}\label{corggg}
	Let $\h$ be a connected $k$-graph. If $\x$ is its signless Laplacian vector, then \[\frac{\delta}{\Delta}\gamma(\h)\leq\Gamma(\h)\leq \gamma(\h).\]	
	If $\h$ is regular, then the equalities are achieved.
\end{Cor}
\begin{proof} These inequalities follow from Theorem \ref{lemapeso}:
	\[\x(\max) \leq kx_{\max}, \quad \x(\min) \geq kx_{\min}
	\quad\Rightarrow\quad \Gamma(\h)\leq\gamma(\h).\]
	\[\x(\max)\geq \frac{\rho}{\Delta} x_{\max}, \quad \x(\min) \leq \frac{\rho}{\delta} x_{\min}\quad\Rightarrow\quad \Gamma(\h)\geq \frac{\delta}{\Delta}\gamma(\h).\]
	
	Observe that if $\h$ is regular, then $\gamma = 1$, $\Gamma = 1$  and
	$\Delta = \delta$. So the equalities hold.
\end{proof}

\begin{Obs}
	We notice that, compute the values of $x_{\max}$, $x_{\min}$ and $\gamma$ are easier than to obtain the parameters $\x(\max)$, $\x(\min)$ and $\Gamma$. Thus, Theorem \ref{lemapeso} and Corollary \ref{corggg} provide useful bounds for these new parameters. Moreover, an interpretation of these results suggests a strong relationship between the parameters for vertices and edges. That is, the information passed by each parameter set certainly is not the same, but in many cases it must be similar. This reinforces the idea that the values $\x(\max)$, $\x(\min)$ and $\Gamma$ play the role for the edges of the hypergraph, played by $x_{\max}$, $x_{\min}$ and $\gamma$ for vertices.
\end{Obs}

Obviously if $\h$ is a regular $k$-graph, then $\gamma = 1\; \Rightarrow\; x_{\max} =
x_{\min} \; \Rightarrow\; \x(\max) = \x(\min) \; \Rightarrow\; \Gamma =1$. But
it is not true that if $\Gamma = 1$ then $\h$ is regular. Thus the following
question naturally arise.

\begin{Que}
	For which hypergraphs $\h$ do we have $\Gamma(\h) = 1 $?
\end{Que}

We will answer this question in the next theorem.

\begin{Teo}\label{teo:caracG} Let $\h$ be a connected $k$-graph. $\Gamma(\h) = 1$ if, and only if, for each edge the sum of the degrees of its vertices is constant.
\end{Teo}

\begin{proof} If $\Gamma(\h) = 1$ then $\x(\min) = \x(\max)$, so by Theorem
	\ref{Teominrhormax}, we have $\x(e) = \sqrt{\frac{\rho}{m}}$, for all $e=\{v_1,\ldots,v_k\}\in E$. For
	each $u \in V$, we have \[\rho x_u = \sum_{e \in E_{[u]}}\x(e) =
	d(u)\sqrt{\frac{\rho}{m}}\;\Rightarrow\; x_u = \frac{d(u)}{\sqrt{\rho m}}.\]
	Therefore, \[\sqrt{\frac{\rho}{m}} = \x(e) = \frac{d(v_1)+\cdots+d(v_k)}{\sqrt{\rho m}}
	\; \Rightarrow\; \rho = d(v_1)+\cdots +d(v_k),\;\forall
	e=\{v_1,\ldots,v_k\} \in E.\]
	That is, for any edge, the sum of the degrees of its vertices is always $\rho$.
	
	\vspace{0.2cm}
	
	Conversely, if $d(v_1)+\cdots+d(v_k) = D$ for all $e=\{v_1,\ldots,v_k\} \in E$, we define the vector $\x$ by $x_u = \frac{d(u)}{\sqrt{\rho m}}$ for all $u \in V$, and then
	\[\sum_{e \in E_{[u]}}\x(e) = \sum_{e \in E_{[u]}} \frac{d(v_1)+\cdots+d(v_k)}{\sqrt{\rho m}} = \sum_{e \in 		E_{[u]}} \frac{D}{\sqrt{\rho m}} =  d(u)\frac{D}{\sqrt{\rho m}} = D x_u.\]	
	Hence, $(\Q\x)_u = D x_u$ for all $u \in V$. That is,
	$(D,\x)$ is an eigenpair of $\h$, and by Perron-Frobenius Theorem, we
	know that $\x$ is the signless Laplacian vector of $\h$, so $\x(\max) =
	\sqrt{\frac{\rho}{m}} = \x(\min)$ and therefore $\Gamma(\h) = 1$.
\end{proof}

If we define that a uniform hypergraph is \textit{regular edge}, when for each edge the sum of the degrees of its vertices is constant. It is reasonable to say that the parameter $\Gamma$ measures how far a hypergraph is from being regular edge. Since if $\Gamma$ is greater than 1, then the sum of the vertex degrees of each edge is not constant.

\begin{Exe}
	Suppose $\g$ is a connected graph with $\Gamma = 1$. Let $e = \{u_1, u_2\} \in E$, such that $d(u_1) = d_1$ and $d(u_2) = d_2$. Since the sum of the vertex degrees of each edge is constant, then all neighbors of $u_1$ must have degree $d_2$, as well as, all neighbors of $u_2$ must have degree $d_1$, by the graph's connectivity, all edges must have a vertex with degree $d_1$ and another with degree $d_2$. Therefore $\g$ is regular or bipartite semi-regular.
\end{Exe}

\begin{Def}
	Let $\h=(V,E)$ be a $k$-graph, let $s \geq 1$ and $r \geq ks$ be
	integers. We define the (generalized) \textit{power hypergraph} $\h^r_s$ as
	the $r$-graph with the following sets of vertices and edges
	
	$$V(\h^r_s)=\left( \bigcup_{v\in V} \varsigma_v\right) \cup \left(
	\bigcup_{e\in E} \varsigma_e\right)\;\; \textrm{and}\;\;
	E(\h^r_s)=\{\varsigma_e\cup \varsigma_{v_1} \cup \cdots \cup \varsigma_{v_k}
	\colon e=\{v_1,\ldots, v_k\} \in E\},$$
		
	\noindent where $\varsigma_{v}=\{v_{1}, \ldots,
	v_{s}\}$ for each vertex $v \in V(\h)$ and $\varsigma_e=\{v^1_e,
	\ldots,v^{r-ks}_e\}$ for each edge $e \in E(\h)$.
\end{Def}

Informally, we may say that $\h^r_s$ is obtained from a \textit{base hypergraph}
$\h$, by replacing each vertex $v\in V(\h)$ by a set $\varsigma_v$ of
cardinality $s$, and by adding a set $\varsigma_e$ with $r-ks$ new vertices
to each edge $e \in E(\h)$. For more details about this class, see \cite{Kaue-power}.

\begin{Exe}
	The power hypergraph $(P_4)^5_2$ of the path with four vertices $P_4$ is illustrated in
	Figure \ref{fig:ex1}.
	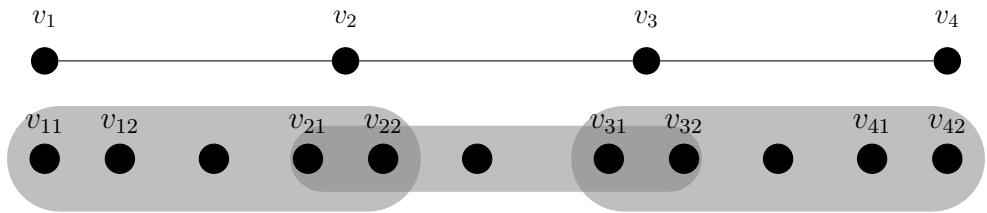
\begin{figure}[h!]
		\centering
		\begin{tikzpicture}
		[scale=1,auto=left,every node/.style={circle,scale=0.9}]
		\node[draw,circle,fill=black,label=below:,label=above:\(v_1\)] (v1) at (0,0) {};
		\node[draw,circle,fill=black,label=below:,label=above:\(v_2\)] (v2) at (4,0) {};
		\node[draw,circle,fill=black,label=below:,label=above:\(v_3\)] (v3) at (8,0) {};
		\node[draw,circle,fill=black,label=below:,label=above:\(v_4\)] (v4) at (12,0) {};
		\path
		(v1) edge node[left]{} (v2)
		(v2) edge node[below]{} (v3)
		(v3) edge node[left]{} (v4);
		\end{tikzpicture}
		
		\begin{tikzpicture}
		\node[draw,circle,fill=black,label=below:,label=above:\(v_{11}\)] (v1) at (0,0) {};
		\node[draw,circle,fill=black,label=below:,label=above:\(v_{12}\)] (v11) at (1,0) {};
		\node[draw,circle,fill=black,label=below:,label=above:] (v22) at (2.25,0) {};
		
		\node[draw,circle,fill=black,label=below:,label=above:\(v_{21}\)] (v3) at (3.5,0) {};
		\node[draw,circle,fill=black,label=below:,label=above:\(v_{22}\)] (v31) at (4.5,0) {};
		\node[draw,circle,fill=black,label=below:,label=above:] (v42) at (5.75,0) {};
		
		\node[draw,circle,fill=black,label=below:,label=above:\(v_{31}\)] (v5) at (7.5,0) {};
		\node[draw,circle,fill=black,label=below:,label=above:\(v_{32}\)] (v51) at (8.5,0) {};
		\node[draw,circle,fill=black,label=below:,label=above:] (v62) at (9.75,0) {};
		
		\node[draw,circle,fill=black,label=below:,label=above:\(v_{41}\)] (v64) at (11,0) {};
		\node[draw,circle,fill=black,label=below:,label=above:\(v_{42}\)] (v7) at (12,0) {};

		\begin{pgfonlayer}{background}
		\draw[edge,color=gray] (v1) -- (v31);
		\draw[edge,color=gray,line width=25pt] (v3) -- (v51);
		\draw[edge,color=gray] (v5) -- (v7);
		\end{pgfonlayer}
		\end{tikzpicture}
		\caption{The power hypergraph $(P_4)^5_2$.}
		\label{fig:ex3}\label{fig:ex1}
	\end{figure}
\end{Exe}

\begin{Teo}
	Let $\h$ be a connected $k$-graph. If $\Gamma(\h) = 1$, then $\Gamma(\h^r_s) = 1.$
\end{Teo}
\begin{proof}
	Just notice that, if the sum of the degree of vertices of each edge of $\h$ is constant, then the sum should remains constant in the edges of $\h^r_s$.
\end{proof}

The bounds of Theorem \ref{teo:qcotas} has already been proved in \cite{Kaue-lap}. However, the technique used did not allow to determine for which class of hypergraphs the equality is achieved. We will present a new proof for the result, determining when equalities hold.

\begin{Teo}\label{teo:qcotas}
	If $\h$ is a connected $k$-graph and $(\rho,\x)$ is its signless Laplacian eigenpair, then
	\[\min_{e\in E}\left\lbrace \sum_{v \in e}d(v)\right\rbrace  \leq \rho(\Q) \leq \max_{e\in E}\left\lbrace \sum_{v \in e}d(v)\right\rbrace.\]
	
	Any equality occurs if, and only if, $\h$ is regular edge.
\end{Teo}
\begin{proof}
	Let $e=\{v_1,\ldots,v_k\}$ be an edge such that $\x(e) = \x(\max)$, we have that
	
	\begin{eqnarray}\rho (x_{v_1}+\cdots+x_{v_k}) &=& \sum_{\alpha \in E_{[v_1]}}\x(\alpha)+\cdots+\sum_{\alpha \in E_{[v_k]}}\x(\alpha)\notag\\
	\rho \x(\max)&\leq&\sum_{\alpha \in E_{[v_1]}}\x(\max)+\cdots+\sum_{\alpha \in E_{[v_k]}}\x(\max)\notag\\
	\rho \x(\max)&\leq& (d(v_1)+\cdots+d(v_k))\x(\max)\notag\\
	\rho &\leq& d(v_1)+\cdots+d(v_k) \leq \max_{e\in E}\left\lbrace \sum_{v \in e}d(v)\right\rbrace. \notag
	\end{eqnarray}
	Observe that, if $\h$ is regular edge then the equality holds. If the equality occurs, then $\x(\alpha) = \x(\max)$ for all edge in $E_{[v_1]}\cup\cdots\cup E_{[v_k]}$,  by the graph's connectivity we have that $\x(\alpha) = \x(\max)$ for all $\alpha \in E(\h)$, that is $\h$ is regular edge.
	
	Similarly, we obtain the lower bound.
\end{proof}


\bibliographystyle{acm}
\bibliography{Bibliografia}

\end{document}